\documentclass[11pt]{article} 

\usepackage[utf8]{inputenc} 

\usepackage{geometry} 
\usepackage{booktabs} 
\usepackage{array} 
\usepackage{paralist} 
\usepackage{verbatim} 
\usepackage{subfig} 
\usepackage{amsmath}
\usepackage{amsfonts}
\usepackage{amssymb}
\usepackage{graphicx} 

\usepackage{float}
\floatstyle{boxed} 
\restylefloat{figure}

\usepackage{fancyhdr} 
\usepackage{sectsty}
\allsectionsfont{\sffamily\mdseries\upshape} 

\usepackage[nottoc,notlof,notlot]{tocbibind} 
\usepackage[titles,subfigure]{tocloft} 

\newtheorem{theorem}{Theorem}[section]
\newtheorem{lemma}[theorem]{Lemma}

\newtheorem{corollary}[theorem]{Corollary}
\newtheorem{conjecture}[theorem]{Conjecture}

\newenvironment{proof}[1][Proof]{\begin{trivlist}
\item[\hskip \labelsep {\bfseries #1}]}{\end{trivlist}}
\newenvironment{definition}[1][Definition]{\begin{trivlist}
\item[\hskip \labelsep {\bfseries #1}]}{\end{trivlist}}

\newcommand{\qed}{\nobreak \ifvmode \relax \else
\ifdim\lastskip<1.5em \hskip-\lastskip
\hskip1.5em plus0em minus0.5em \fi \nobreak
\vrule height0.75em width0.5em depth0.25em\fi}

\title{A Conjecture concerning the Fibonomial Triangle}
\author{Jeremiah Southwick}

\begin{document}
\maketitle

\begin{abstract}
The fibonomial triangle has been shown by Chen and Sagan to have a fractal nature mod 2 and 3. Both these primes have the property that the Fibonacci entry point of $p$ is $p+1$. We study the fibonomial triangle mod 5, showing with a theorem of Knuth and Wilf that the triangle has a recurring structure under divisibility by five. While this result is not new, our method of proof is new and suggests a conjecture for the divisibility of a fibonomial coefficient by a general prime $p$. We give necessary conditions for such primes, namely that the Fibonacci entry point must be greater than or equal to $p$, and offer numerical evidence for the validity of the conjecture. Lastly, we conclude with a discussion concerning further directions of research.
\end{abstract}

\section{Introduction}
Pascal's triangle is constructed from binomial coefficients ${n \choose k} = \frac{n!}{k!(n-k)!}$, where ${n \choose k}$ is the $k^{th}$ entry of the $n^{th}$ row (See Figure \ref{PasTri}). It is well-known that Pascal's triangle exhibits a fractal nature mod $p$ where $p$ is a prime. For example, in the mod 2 case, the triangle follows the pattern of Sierpinski's fractal, the fractal which begins with a triangle and removes the center quarter, doing the same for each remaining upright triangle at each iteration, as in Figure \ref{SierIter}.

\begin{figure}
\caption{Pascal's Triangle}
\label{PasTri}
\begin{center}
$
\begin{array}{ccccccccccccccc}
& & & & & & & 1 & & & & & & & \\
& & & & & & 1 & & 1 & & & & & & \\
& & & & & 1 & & 2 & & 1 & & & & & \\
& & & & 1 & & 3 & & 3 & & 1 & & & & \\
& & & 1& & 4 & & 6 & & 4 & & 1 & & & \\
& & 1 & & 5 & & 10 & & 10 & & 5 & & 1 & & \\
& 1 & & 6 & & 15 & & 20 & & 15 & & 6 & & 1 & \\
1 & & 7 & & 21 & & 35 & & 35 & & 21 & & 7 & & 1 \\
\end{array}
$
\end{center}
\end{figure}

\begin{figure}
\caption{Sierpinski's Fractal}
\label{SierIter}
\begin{center}
\includegraphics[scale=0.5]{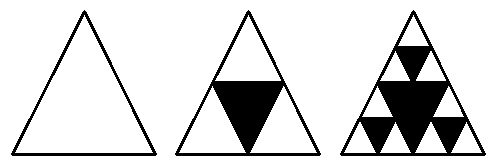}
\end{center}
\end{figure}

One can correlate each iteration of Sierpinski's fractal to the first $2^n$ rows of Pascal's triangle mod 2, where the ``removed'' portion is all the zeros mod 2, that is, all the even numbers. The first 8 rows of this triangle, corresponding to the second iteration of Sierpinski's fractal, are shown in Figure \ref{Pascal2}.

\begin{figure}
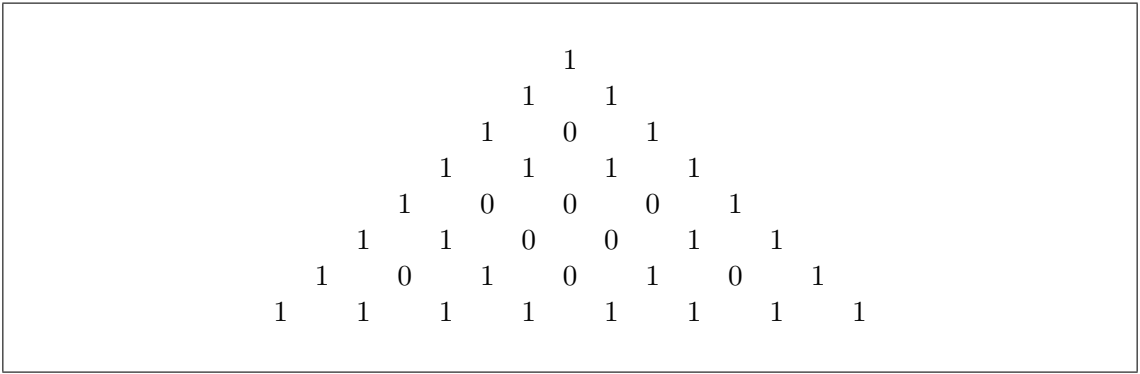

\caption{Pascal's Triangle mod 2}
\label{Pascal2}
\begin{center}
$
\begin{array}{ccccccccccccccc}
& & & & & & & 1 & & & & & & & \\
& & & & & & 1 & & 1 & & & & & & \\
& & & & & 1 & & 0 & & 1 & & & & & \\
& & & & 1 & & 1 & & 1 & & 1 & & & & \\
& & & 1& & 0 & & 0 & & 0 & & 1 & & & \\
& & 1 & & 1 & & 0 & & 0 & & 1 & & 1 & & \\
& 1 & & 0 & & 1 & & 0 & & 1 & & 0 & & 1 & \\
1 & & 1 & & 1 & & 1 & & 1 & & 1 & & 1 & & 1 \\
\end{array}
$
\end{center}
\end{figure}

Beyond $p = 2$, the exact fractal produced mod $p$ is different than Sierpinski's, but exhibits the same repetitive structure. This structure can be described mathematically by expressing $n$ and $k$ in base $p$. We will use $(n)_p$ and $(k)_p$ to notate this.

\begin{definition}
Let $n = n_0 + n_1 \cdot p + n_2 \cdot p^2 + \cdots$ where $0 \leq n_i < p$ for all $i$. Then we say $(n)_p = (n_0 \ n_1 \ n_2 \ \cdots)_p$.
\end{definition}

When the context is unclear, we may use the subscript $p$ on either $n$ or its expansion. For example, when $p = 7$, we write $109 = (4)1 + (1)7 + (2)49 = (4 \ 1 \ 2)_7$. The pertinent result is a theorem of Lucas:

\begin{theorem}[Lucas]
\label{lucas}
Let $p$ be a prime and $(n)_p = (n_i)_{i \geq 0}, (k)_p = (k_i)_{i \geq 0}$. Then

\begin{center}
${n \choose k} \equiv_p {n_0 \choose k_0}{n_1 \choose k_1}{n_2 \choose k_2} \cdots$
\end{center}
\end{theorem}

We have reproduced Theorem \ref{lucas} as it appears in \cite{CheSag14}. A proof of the theorem can be found in \cite{Fin47}. What the theorem essentially means is that the exact residue class of any binomial coefficient mod $p$ can be determined by considering only binomials in the first $p$ rows of Pascal's triangle, since the $n_i$ are between 0 and $p-1$.

Binomial coefficients can be generalized to produce other triangles, which may also be studied mod $p$. In this paper we will specifically study the fibonomial triangle, which is constructed from the Fibonacci sequence.

Let $F_1 = 1, F_2 = 1$ and $F_n = F_{n-1} + F_{n-2}$ for $n \geq 3$, so that $(F_n)$ is the Fibonacci sequence $(1, 1, 2, 3, 5, 8, 13, 21, \cdots).$ There are many combinatorial interpretations of these numbers; for example, $F_n$ counts the number of ways to tile a row of $n-1$ tiles with dominos and squares (see \cite{BenQui03}). We will use this result in a later proof. The sequence also appears in various places throughout nature, such as in pinecone spirals and the unencumbered growth of rabbit populations. We define a generalization of factorials in terms of $(F_n)$, appropriately named fibotorials. We follow \cite{CheSag14} in our notation by denoting this number as $n!_F$.

\begin{definition}
For $n > 0$, let $n!_F = F_n \cdot F_{n-1} \cdots F_1$. We define $0!_F$ to be 1.
\end{definition}

For example, $6!_F = 8 \cdot 5 \cdot 3 \cdot 2 \cdot 1 \cdot 1 = 240$. Having defined a generalization of factorials, we can now give the definition of fibonomial coefficients. These are defined identically to binomial coefficients, except with fibotorials instead of factorials:

\begin{definition}
The fibonomial coefficient ${n \choose k}_F = \frac{n!_F}{k!_F(n-k)!_F}$ for $0 \leq k \leq n$. Otherwise, ${n \choose k}_F = 0$.
\end{definition}

Thus ${5 \choose 2}_F = \frac{5 \cdot 3 \cdot 2 \cdot 1 \cdot 1}{(2 \cdot 1 \cdot 1)(1 \cdot 1)} = 15$ and ${6 \choose 3}_F = \frac{240}{(2 \cdot 1 \cdot 1)(2 \cdot 1 \cdot 1)} = 60$. It would be reasonable to suspect that ${n \choose k}_F$ has a combinatorial interpretation, and in fact it does. Given a $k\times (n-k)$ board with a lattice path $L$ running from the lower left corner to the upper right corner, ${n \choose k}_F$ counts the number of ways over all possible paths $L$ to tile the rows above $L$ with dominos and squares while tiling the columns below $L$ with tilings of dominos and squares, with the added proviso that each column must have a domino in its first two (lowest) tiles. A proof of this fact is given in \cite{SagSav10}. Thus ${n \choose k}_F$ is an integer for all $n, k$.

Fibonomial coefficients satisfy an analogue of the binomial relation ${n \choose k} = {n-1 \choose k} + {n-1 \choose k-1}$, in that ${n \choose k}_F = F_{k+1}{n-1 \choose k}_F + F_{n-k-1}{n-1 \choose k-1}_F$. Benjamin and Reiland \cite{BenRei14} offer a combinatorial proof of this fact by breaking the tilings into two disjoint sets, the first being tilings of lattice paths ending with a vertical step, of which there are $F_{k+1}{n-1 \choose k}_F$, and the second being tilings of lattice paths ending with a horizontal step, of which there are $F_{n-k-1}{n-1 \choose k-1}_F$.

Using this recurrence or the definition, one can generate the fibonomial triangle. We have reproduced the first 8 rows in Figure \ref{FibTri}. Given this triangle, we would like to derive conditions similar to Theorem \ref{lucas} for when a prime $p$ divides ${n \choose k}_F$. Such a coefficient exists, since there is a first Fibonacci number which $p$ divides. We follow \cite{CheSag14} in our notation of this number.

\begin{figure}
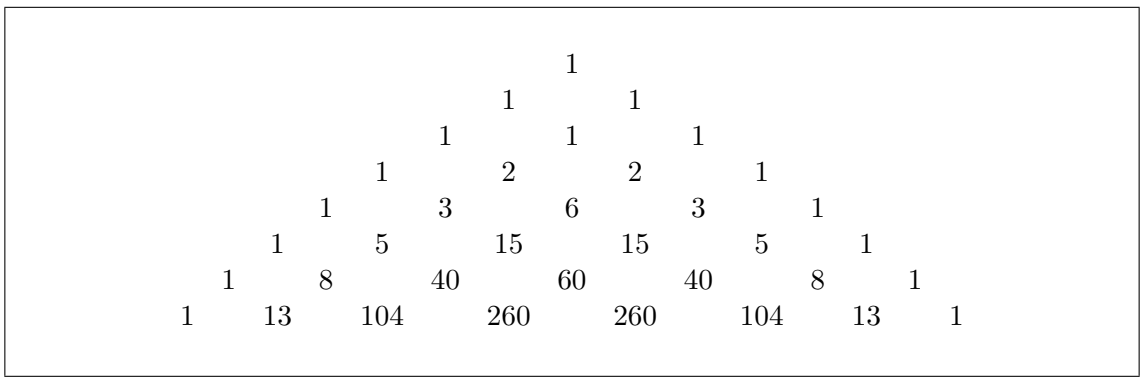

\caption{The Fibonomial Triangle}
\label{FibTri}
\begin{center}
$
\begin{array}{ccccccccccccccc}
& & & & & & & 1 & & & & & & & \\
& & & & & & 1 & & 1 & & & & & & \\
& & & & & 1 & & 1 & & 1 & & & & & \\
& & & & 1 & & 2 & & 2 & & 1 & & & & \\
& & & 1& & 3 & & 6 & & 3 & & 1 & & & \\
& & 1 & & 5 & & 15 & & 15 & & 5 & & 1 & & \\
& 1 & & 8 & & 40 & & 60 & & 40 & & 8 & & 1 & \\
1 & & 13 & & 104 & & 260 & & 260 & & 104 & & 13 & & 1 \\
\end{array}
$
\end{center}
\end{figure}

\begin{definition}
Given a prime number $p$, we define $p^*$ to be the least positive integer $n$ such that $p \mid F_n$.
\end{definition}

For example, $11^* = 10$, since $F_{10} = 55$ is the first Fibonacci number which 11 divides. There are several different conventions for naming $p^*$. In \cite{KnuWil89} it is referred to as $r(p)$, \textit{the rank of apparition of $p$}. In \cite{CubRou14} it is referred to as $Z(p)$, \textit{the Fibonacci entry point}. Throughout the paper we will simply use $p^*$.

The existence of $p^*$ gives a new base to consider when using the Fibonacci numbers.

\begin{definition}
Let $\mathcal{F}_p = (1, p^*, p^*p, p^*p^2, \cdots)$.
\end{definition}

$\mathcal{F}_p$ has all the usual properties of a base, since $1 \mid p^*$ and $p^*p^i \mid p^*p^{i+1}$ for all $i = 0, 1, \cdots$. Thus every number $n$ can be written uniquely as $n = n_0(1) + n_1(p^*) + n_2(p^*p) + \cdots = (n_0 \ n_1 \ n_2 \ \cdots)_{\mathcal{F}_p}$ where $0 \leq n_0 < p^*$ and $0 \leq n_i < p$ for all other $n_i$. This is a natural base to use for the Fibonacci numbers, since $p^* = (0 \ 1)_{\mathcal{F}_p}$, just as in base $p$ we have $p = (0 \ 1)_p$.

One might wonder naively whether this base could be used to show a similar congruence to Theorem \ref{lucas}, namely whether the congruence

\begin{center}
${n \choose k}_F \equiv_p {n_0 \choose k_0}_F {n_1 \choose k_1}_F {n_2 \choose k_2}_F \cdots$
\end{center}

\noindent holds for all primes $p$. But this is not the case. As shown in \cite{CheSag14}, the congruence holds when $p = 2$, but for $p = 3$ the relationship only maintains divisibility by 3. In fact, we will show that whenever $p^* < p$, the relationship does not even preserve divisibility by $p$. In contrast to these cases, if $p^* \geq p$, it appears that divisibility by $p$ is maintained. We thus introduce the following conjecture:

\begin{conjecture}
\label{OpenProb}
Let $p^* \geq p$ and let $\mathcal{F}_p$ be as above. Express $(n)_{\mathcal{F}_p} = (n_i)_{i\geq0}$ and $(k)_{\mathcal{F}_p} = (k_i)_{i\geq0}$. Then

\begin{center}
$p \mid {n \choose k}_F \Leftrightarrow p \mid {n_0 \choose k_0}_F {n_1 \choose k_1}_F {n_2 \choose k_2}_F \cdots$.
\end{center}
\end{conjecture}

In fact, since $F_3 = 2$ and $F_4 = 3$, we have that $2^* = 3$ and $3^* = 4$ (so $2^* > 2$ and $3^* > 3$), meaning Conjecture \ref{OpenProb} has already been shown for the first two cases in \cite{CheSag14}.

The rest of the paper will be devoted to studying the validity of Conjecture \ref{OpenProb}. Section 2 will detail the important results to consider in the literature, as well as the techniques that will aid us in our endeavors. In Section 3 we will show that Conjecture \ref{OpenProb} is true when $p = 5$, and will prove an immediate corrollary. Section 4 will prove the necessary conditions for Conjecture \ref{OpenProb}, namely that it fails when $p^* < p$. Section 5 concludes with a discussion of possible directions for future research.

\section{Useful Results in the Literature}

\subsection{The fibonomial triangle mod 2 and 3}

Chen and Sagan \cite{CheSag14} have studied the fractal nature of the fibonomial triangle mod 2 and 3, deriving the exact residue class of ${n \choose k}_F$ in each case for all $n$ and $k$. Their main theorem in the mod 2 case is as follows:

\begin{theorem}[Chen and Sagan \cite{CheSag14}]
\label{CheSagMain}
Given $m \geq 0$ and $0 \leq n, k < 3 \cdot 2^m$ we have $${n + 3 \cdot 2^m \choose k}_F \equiv_2 {n \choose k}_F.$$
\end{theorem}

What this essentially means is that the fibonomial triangle mod 2 is quite similar to Pascal's triangle mod 2, in that it correlates to Sierpinski's fractal, except with an initial triangle of height three, rather than height 2. See figure \ref{FibTri2}. Chen and Sagan's primary method of proof is combinatorial. Using the tiling interpretation of ${n \choose k}_F$, they pair tilings together by means of an involution, and determine the residue class based on whether or not there is an unpaired tiling left over. As secondary methods, they also show Theorem \ref{CheSagMain} from a number theoretic perspective, and then with an inductive argument.

\begin{figure}
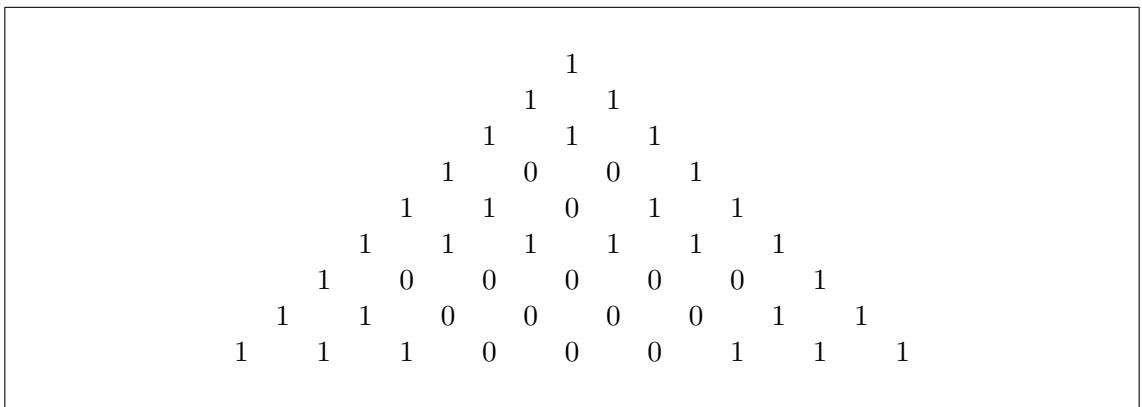

\caption{The Fibonomial Triangle mod 2}
\label{FibTri2}
\begin{center}
$
\begin{array}{ccccccccccccccccc}
& & & & & & & & 1 & & & & & & & & \\
& & & & & & & 1 & & 1 & & & & & & & \\
& & & & & & 1 & & 1 & & 1 & & & & & & \\
& & & & & 1 & & 0 & & 0 & & 1 & & & & & \\
& & & & 1& & 1 & & 0 & & 1 & & 1 & & & & \\
& & & 1 & & 1 & & 1 & & 1 & & 1 & & 1 & & & \\
& & 1 & & 0 & & 0 & & 0 & & 0 & & 0 & & 1 & & \\
& 1 & & 1 & & 0 & & 0 & & 0 & & 0 & & 1 & & 1 & \\
1 & & 1 & & 1 & & 0 & & 0 & & 0 & & 1 & & 1 & & 1 \\
\end{array}
$
\end{center}
\end{figure}

We will make use of methods similar to Chen and Sagan's number theoretic proof of Theorem \ref{CheSagMain}. While making this argument, they use the base $\textbf{F} = (1, 3, 3 \cdot 2, 3 \cdot 2^2, \cdots)$. Since $2^* = 3$, this is identical to the base $\mathcal{F}_2$ we defined in the introduction. Using this base, they offer the following theorem:

\begin{theorem}[Chen and Sagan \cite{CheSag14}]
Let $(n)_{\textbf{F}} = (n_i)_{i \geq 0}$ and $(k)_{\textbf{F}} = (k_i)_{i \geq 0}$. Then

\begin{center}
${n \choose k}_F \equiv_2 {n_0 \choose k_0}_F{n_1 \choose k_1}_F{n_2 \choose k_2}_F \cdots$.
\end{center}
\end{theorem}

As noted in the introduction, since $\textbf{F} = \mathcal{F}_2$, this shows Conjecture \ref{OpenProb} for when $p = 2$. Chen and Sagan prove their theorem with a case-by-case analysis of the congruence classes of $n$ and $k$, relying on a theorem of Knuth and Wilf involving the $p$-adic valuation of ${n \choose k}_F$. We will use this theorem in Section 3 to show Conjecture \ref{OpenProb} holds when $p = 5$.

\subsection{$p$-adic valuations of fibonomial coefficients}

Knuth and Wilf \cite{KnuWil89} study generalized binomial coefficients and derive conditions for when the prime $p$ divides ${n \choose k}_F$. Their main theorem regarding fibonomials is concerned with the $p$-adic valuation of ${n \choose k}_F$, that is, the highest power of $p$ dividing ${n \choose k}_F$.

\begin{definition}
Let $x \in \mathbb{N}$, and let $p$ be a prime. Then $\nu_p(x)$ is the $p$-valuation of $x$, i.e., the highest power of $p$ dividing $x$.
\end{definition}

The simplest way to calculate $\nu_p({n \choose k}_F)$ is to calculate the $p$-valuation of the numerator and subtract from this the $p$-valuation of the denominator. For example, $\nu_3({5 \choose 2}_F) = \nu_3(\frac{5\cdot3\cdot2\cdot1\cdot1}{(2\cdot1\cdot1)(1\cdot1)}) = \nu_3(30) - [\nu_3(2) + \nu_3(1)] = 1 - 0 - 0 = 1$. These definitions give us enough information to state the pertinent theorem of Knuth and Wilf:

\begin{theorem}[Knuth and Wilf \cite{KnuWil89}]
\label{KnuWil}
The highest power of the odd prime $p$ that divides the fibonomial coefficient ${m + n \choose m}_F$ is the number of carries that occur to the left of the radix point when $m/p^*$ is added to $n/p^*$ in $p$-ary notation, plus $\nu_p(F_{p^*})$ if a carry occurs across the radix point.
\end{theorem}

For example, when $p = 7$, we have $7^* = 8$ since $F_8 = 21$ is the first Fibonacci number which 7 divides. Additionallly, we have $\nu_7(F_8) = 1$. Then given ${57 \choose 26}_F = {26+31 \choose 26}_F$, we have $26 = 3 \cdot 7 + 5 \cdot 1 = (3 \ 5)_7$ and $31 = 4 \cdot 7 + 3 \cdot 1 = (4 \ 3)_7$. To divide by 8 in base 7, we must express 8 as $8 = (1 \ 1)_7$, which gives $35/11 = 3.\overline{1515}$ and $43/11 = 3.\overline{6060}$. Then adding gives

\begin{center}
$
\begin{array}{cc}
& 3.\overline{1515} \\
+ & 3.\overline{6060} \\ \hline
= & 10.\overline{0606} \\
\end{array}
$
\end{center}

\noindent where one carry occurs across the radix point and one carry occurs from the 1's place to the 7's place. Thus by Theorem \ref{KnuWil}, $\nu_7({57 \choose 26}_F) = \nu_7(F_8) + 1 = 1 + 1 = 2$.

We wish to use Theorem \ref{KnuWil} to show that Conjecture \ref{OpenProb} holds when $p = 5$, as these methods illustrate how the conjecture might be proven in the case of a more general prime $p$. For our purposes in the $p=5$ case, we need the fact that $5^* = 5$, since the first Fibonacci number divisible by 5 is $F_5$. Then $\nu_5(F_{5^*}) = \nu_5(F_5) = \nu_5(5) = 1$. Using this fact, we will prove the following corollary of Knuth and Wilf's theorem:

\begin{corollary}
\label{FiveVal}
The highest power of 5 that divides the fibonomial coefficient ${m + n \choose m}_F$ is the number of carries that occur when $m/5$ is added to $n/5$ in 5-ary notation.
\end{corollary}

\begin{proof}
Break the carries down into three types: those occurring to the left of the radix point, a carry occurring across it, and those occurring to the right of the radix point. Theorem \ref{KnuWil} simply counts the number of the first type of carry, as does this corollary. For a carry occurring across the radix point, Theorem \ref{KnuWil} counts $\nu_p(F_{p^*})$. But we have $\nu_5(F_{5^*}) = 1$, so we can simply count whether or not this carry occurs. Lastly, Theorem \ref{KnuWil} doesn't count carries to the right of the radix point. But no such carries occur for our corollary: Since division by 5 in base 5 simply moves the radix point one space to the left (identical to division by 10 in base 10), the numbers $m/5$ and $n/5$ each terminate one place to the right of the radix point, thus eliminating the possibility of any carries further to the right of the radix point. So counting the total number of carries is identical to counting the number of carries described in Theorem \ref{KnuWil}. $\qed$
\end{proof}

\subsection{Divisibility of Fibonacci numbers by a prime $p$}

To give necessary conditions for Conjecture \ref{OpenProb}, we will need several well-known facts about the divisibility of Fibonacci numbers. Among these are 

\begin{center}
$\gcd(F_n, F_m) = F_{\gcd(n,m)}$ and its corollary $F_n \mid F_m \Leftrightarrow n \mid m$,
\end{center}

as well as Binet's formula, which denotes the distinct roots of $x^2 - x - 1$ as $\alpha$ and $\beta$, with the fact that

\begin{center}
$F_{n} = \frac{1}{\sqrt{5}}(\alpha^{n} - \beta^{n})$.
\end{center}

We must also introduce the Lucas numbers $(L_n) = (2, 1, 3, 4, 7, \cdots)$, which have the similar relation

\begin{center}
$L_n = \alpha^n + \beta^n$.
\end{center}

In \cite{HogBer74}, Hoggatt and Bergum make use of these facts to prove a theorem comparing the divisibility by $p$ of $F_n$ with that of $F_{np^k}$:

\begin{theorem}[Hoggatt and Bergum \cite{HogBer74}]
\label{berhog}
If $p$ is an odd prime and $p \mid F_n$ then $p^k \mid F_{np^{k-1}}$ for all $k \geq 1.$
\end{theorem}

We desire a slightly stronger result, namely, that the $p$-adic valuation increases by exactly 1 when the subscript of a Fibonacci number divisible by $p$ is multipled by $p$. This is in fact true, and we can generalize the proof of Theorem \ref{berhog} to prove the desired result:

\begin{lemma}
\label{valinc}
Let $p$ be an odd prime. Then if $\nu_p(F_n) = k > 0$, we have $\nu_p(F_{np}) = k+1$.
\end{lemma}

\begin{proof}
Assume that $\nu_p(F_n) = k$ for some positive $k$. We wish to show that $\nu_p(F_{np}) = k+1$. Using Binet's formula, we have 

\begin{center}
$F_{np} = \frac{1}{\sqrt{5}}(\alpha^{np} - \beta^{np})$.
\end{center}

Then by factoring $\alpha^{np} - \beta^{np}$, we arrive at

\begin{center}
$F_{np} = \frac{1}{\sqrt{5}}(\alpha^{np} - \beta^{np}) = \frac{1}{\sqrt{5}}(\alpha^n - \beta^n)(\sum_{i = 1}^p \alpha^{n(p-i)}\beta^{n(i-1)})$

$=F_n(\sum_{i = 1}^p \alpha^{n(p-i)}\beta^{n(i-1)})$.
\end{center}

In the summation $\sum_{i = 1}^p \alpha^{n(p-i)}\beta^{n(i-1)}$, the middle term is $(-1)^{n(p-1)/2}$ since $\alpha \beta = -1$. As in \cite{HogBer74} we can show that the sum of the $i^{th}$ and $(p+1-i)^{th}$ terms, where $i \not = (p+1)/2$, is

\begin{center}
$\alpha^{n(p-i)}\beta^{n(i-1)} + \alpha^{n(i-1)}\beta^{n(p-i)} = (\alpha \beta)^{n(i-1)}(\alpha^{n(p-2i+1)} + \beta^{n(p-2i+1)})$

$= (-1)^{n(i-1)}(\alpha^{2n(p-2i+1)/2} + \beta^{2n(p-2i+1)/2}) =(-1)^{n(i-1)}L_{2n(p-2i+1)/2}$.
\end{center}

We follow \cite{HogBer74} in writing the subscript this way because of the relation

\begin{center}
$L_{2r} = \alpha^{2r} + \beta^{2r} = \alpha^{2r} - 2(\alpha \beta)^r + \beta^{2r} + 2(-1)^r = (\alpha^r - \beta^r)^2 + 2(-1)^r$

$=5\frac{(\alpha^r - \beta^r)^2}{(\sqrt{5})^2} + 2(-1)^r = 5F^2_r + 2(-1)^r$.
\end{center}

Thus we have that the sum of the two terms is

\begin{center}
$(-1)^{n(i-1)}L_{2n(p-2i+1)/2}=(-1)^{n(i-1)}5F^2_{n(p-2i+1)/2} + 2(-1)^{n(2i-2+p-2i+1)/2}$

$= (-1)^{n(i-1)}5F^2_{n(p-2i+1)/2} + 2(-1)^{n(p-1)/2}$.
\end{center}

Grouping these terms together in the summation, we arrive at

\begin{center}
$F_{np} = F_n \Big{(}\sum_{i=1}^{(p-1)/2} (-1)^{n(i-1)}5F^2_{n(p-2i+1)/2} + p(-1)^{n(p-1)/2}\Big{)}$.
\end{center}

Now using our hypothesis that $\nu_p(F_n) = k$, let us inspect the term in parentheses. Each Fibonacci number in the summation has a subscript that $n$ divides, giving that $F_n$ divides each summand and hence $p^k$ divides each summand. But then the whole term takes the form $p^kq \pm p = p(p^{k-1}q \pm 1)$, so only $p^1$ divides it, and no higher powers. Hence $\nu_p(F_{np}) = \nu_p(F_n) + 1 = k+1$. $\qed$
\end{proof}

\section{Proof of Conjecture \ref{OpenProb} when $p = 5$}

\begin{figure}
\caption{The Fibonomial Triangle mod 5}
\label{FibTri5}
\begin{center}
$
\begin{array}{ccccccccccccccccccc}
& & & & & & & & & 1 & & & & & & & & & \\
& & & & & & & & 1 & & 1 & & & & & & & & \\
& & & & & & & 1 & & 1 & & 1 & & & & & & & \\
& & & & & & 1 & & 2 & & 2 & & 1 & & & & & & \\
& & & & & 1& & 3 & & 1 & & 3 & & 1 & & & & & \\
& & & & 1 & & 0 & & 0 & & 0 & & 0 & & 1 & & & & \\
& & & 1 & & 3 & & 0 & & 0 & & 0 & & 3 & & 1 & & & \\
& & 1 & & 3 & & 4 & & 0 & & 0 & & 4 & & 3 & & 1 & & \\
& 1 & & 1 & & 3 & & 2 & & 0 & & 2 & & 3 & & 1 & & 1 & \\
1 & & 4 & & 4 & & 1 & & 1 & & 1 & & 1 & & 4 & & 4 & & 1 \\
\end{array}
$
\end{center}
\end{figure}

Since we are concerned with the fibonomial triangle mod 5, we have reproduced the first 10 rows of that triangle in Figure \ref{FibTri5}. As noted in \cite{Bal15}, divisibility by 5 of ${n \choose k}_F$ is entirely understood, since the 5-adic valuation of a Fibonacci number ($\nu_5(F_n)$) is the same as the 5-adic valuation of its subscript ($\nu_5(n)$), and hence $\nu_5({n \choose k}_F) = \nu_5({n \choose k})$, since ${n \choose k}$ is obtained by replacing the Fibonacci numbers in ${n \choose k}_F$ with their subscripts. Thus by Theorem \ref{lucas}, Conjecture \ref{OpenProb} holds for $p = 5$. However, we will show this result using Corollary \ref{FiveVal}, to illustrate how the result could be extended to more general primes.

As observed previously, we have $5^* = 5$. Thus the base $\mathcal{F}_5 = (1, 5, 5^2, \cdots)$ is simply base 5. Hence we wish to show the following lemma:

\begin{lemma}
\label{WeightLift}
Let $(n)_5 = (n_i)_{i\geq0}, (n-k)_5 = (m_i)_{i\geq0}$ and $(k)_5 = (k_i)_{i\geq0}$. Then $$5 \mid {n \choose k}_F \Leftrightarrow 5 \mid {n_0 \choose k_0}_F{n_1 \choose k_1}_F{n_2 \choose k_2}_F\cdots.$$
\end{lemma}

\begin{proof}
First note that by our conditions on $n_i$ and $k_i$, for each choice of $n_i$ and $k_i$ we have ${n_i \choose k_i}_F$ appearing in one of the following 25 positions on the fibonomial triangle:

\begin{center}
$
\begin{array}{cccccccccccccccccc}
& & & & & & & & & 1 & & 0 & & 0 & & 0 & & 0 \\
& & & & & & & & 1 & & 1 & & 0 & & 0 & & 0 & \\
& & & & & & & 1 & & 1 & & 1 & & 0 & & 0 & & \\
& & & & & & 1 & & 2 & & 2 & & 1 & & 0 & & & \\
& & & & & 1& & 3 & & 6 & & 3 & & 1 & & & & \\
\end{array}
$
\end{center}

The inverted triangle of zeroes contains the coefficients ${n_i \choose k_i}_F$ with $n_i < k_i$. Thus $5 \mid {n_i \choose k_i}_F \Leftrightarrow n_i < k_i$, since no other multiples of five appear in the 25 positions in question. With that in mind, we proceed with the proof.

$\underline{(\Rightarrow)}:$ By Corollary \ref{FiveVal}, if $5 \mid {n \choose k}_F$ then there is a carry occurring in the addition of $(\frac{k}{5})_5$ and $(\frac{n-k}{5})_5$. But as noted above, division by 5 in base 5 simply moves the radix point one point to the left, so this means there is a carry in the addition of $(k)_5$ and $(n-k)_5$. Let the first carry occur at the $j^{th}$ position, that is, let the addition $k_j + m_j$ have a carry such that for $i < j, k_i + m_i$ has no carry. There are ten possibilities for $n_j, k_j$, and $m_j$, as shown in the table below:

\begin{center}
$
\begin{array}{c|cccccccccc}
n_j & 0 & 1 & 2 & 3 & 0 & 1 & 2 & 0 & 1 & 0 \\ \hline
k_j & 4 & 4 & 4 & 4 & 3 & 3 & 3 & 2 & 2 & 1 \\
m_j & 1 & 2 & 3 & 4 & 2 & 3 & 4 & 3 & 4 & 4 \\
\end{array}
$
\end{center}

Each of these cases corresponds to a fibonomial ${n_j \choose k_j}_F$ that has $n_j < k_j$. But then as noted, we have $5 \mid {n_j \choose k_j}_F$ and consequently $5 \mid {n_0 \choose k_0}_F{n_1 \choose k_1}_F{n_2 \choose k_2}_F\cdots$.

$\underline{(\Leftarrow)}:$ Conversely, let $5 \mid {n_0 \choose k_0}_F{n_1 \choose k_1}_F{n_2 \choose k_2}_F\cdots$. Then $\exists j$ such that $5 \mid {n_j \choose k_j}_F$. Without loss of generality let $j$ be the least such $j$. We will show that $n_j, k_j$, and $m_j$ are as in one of the cases in the above table. First, note that since $5 \mid {n_j \choose k_j}_F$, we have $n_j < k_j$. Thus ${n_j \choose k_j}_F$ is one of the 10 cases above, since those are the only cases where $n_j < k_j$.  Then we can show that $m_j$ follows by minimality of $j$. We are working with the following subtraction problem:

\begin{center}
$
\begin{array}{cc}
& n_0 + 5n_1 + 5^2n_2 + \cdots \\
- & (k_0 + 5k_1 + 5^2k_2 + \cdots)\\
\end{array}
$
\end{center}

For each $i < j, {n_i \choose k_i}_F \not = 0$, since $5 \nmid {n_i \choose k_i}_F$. Then by definition, each ${n_i \choose k_i}_F$ has $n_i \geq k_i$. So there is no carrying involved in calculating $m_i$, which is exactly $n_i - k_i$ for each $i <j$:

\begin{center}
$
\begin{array}{ccccc}
& n_0 & + 5n_1 & + \cdots & + 5^{j-1}n_{j-1} \\
- & k_0 & - 5k_1 & - \cdots & - 5^{j-1}k_{j-1}\\ \hline
= & (n_0-k_0) & + 5(n_1 - k_1) & + \cdots & + 5^{j-1}(n_{j-1} - k_{j-1}) \\
\end{array}
$
\end{center}

Then $m_j$ is simply $5 + n_j - k_j$, where the 5 is carried from $n_{j+1}$. This verifies that $m_j$ is as in the table above. But then a carry occurs in the addition of $(k)_5$ and $(n-k)_5$, meaning a carry occurs in the addition of $(\frac{k}{5})_5$ and $(\frac{n-k}{5})_5$, giving $5 \mid {n \choose k}_F. \qed$
\end{proof}

Thus Conjecture \ref{OpenProb} holds when $p = 5$. Lemma \ref{WeightLift} gives the following corollary, which we include because of its similarity to Theorem \ref{CheSagMain}.

\begin{corollary}
\label{MainCorollary}
Let $0 \leq k, n <5^m$. Then for $ 0 \leq j \leq i \leq 4$, we have $$5 \mid {n + i5^m \choose k + j5^m}_F \Leftrightarrow 5 \mid {n \choose k}_F$$.
\end{corollary}

\begin{proof}
Let $0 \leq k, n < 5^m, 0 \leq j \leq i \leq 4$, and let $(n_i), (k_i)$ be the base 5 expansions of $n$ and $k$, as in Lemma \ref{WeightLift}.

By Lemma \ref{WeightLift}, $5 \mid {n + i5^m \choose k + j5^m}_F \Leftrightarrow 5 \mid {n_0 \choose k_0}_F{n_1 \choose k_1}_F{n_2 \choose k_2}_F\cdots{n_{m-1} \choose k_{m-1}}_F{i \choose j}_F$. Now since $0 \leq j \leq i \leq 4$, we have ${i \choose j}_F \not= 0$ and hence $5 \nmid {i \choose j}_F$.

So $5 \mid {n_0 \choose k_0}_F{n_1 \choose k_1}_F{n_2 \choose k_2}_F\cdots{n_{m-1} \choose k_{m-1}}_F{i \choose j}_F \Leftrightarrow 5 \mid {n_0 \choose k_0}_F{n_1 \choose k_1}_F{n_2 \choose k_2}_F\cdots{n_{m-1} \choose k_{m-1}}_F$.

But then by another direct application of Lemma \ref{WeightLift}, this is true if and only if $5 \mid {n \choose k}_F. \qed$
\end{proof}

\section{Necessary conditions for Conjecture \ref{OpenProb}}

Since Conjecture \ref{OpenProb} holds for the first three primes, it would be reasonable to attempt to show it holds for all primes. For $p = 7$, data suggests that this pattern continues. However, the very next case is a counterexample to this trend. When $p = 11$, we have $11^* = 10$, as noted in the introduction, giving $\mathcal{F}_{11} = (1, 10, 110, 1210, \cdots)$. The fibonomial coefficent ${100 \choose 10}_F$ is not divisible by 11. However, $100 = (0 \ 10)_{\mathcal{F}_{11}}$ and $10 = (0 \ 1)_{\mathcal{F}_{11}}$, so $n_1 = 10$ and $k_1 = 1$, giving $11 \mid {n_1 \choose k_1}_F = {10 \choose 1}_F =55$.

This problem arises because $p^* < p$. We will show that the same problem arises in general when this is the case. Here we will make use of our result from Section 2.

\addtocounter{section}{-2}
\addtocounter{theorem}{5}
\begin{lemma}
\label{valincrestate}
Let $p$ be an odd prime. Then if $\nu_p(F_n) = k > 0$, we have $\nu_p(F_{np}) = k+1$.
\end{lemma}
\addtocounter{section}{2}
\addtocounter{theorem}{-6}

Since $p \mid F_{p^*}$ by definition of $p^*$, we have the following corollary:

\begin{corollary}
\label{neccor}
Given an odd prime $p$, $\nu_p(F_{p^*p}) = \nu_p(F_{p^*}) + 1$.
\end{corollary}

What we have essentially shown is that multiplication by $p$ in the subscript of a Fibonacci number already divisible by $p$ raises the $p$-valuation by one. Now we must show that if $\gcd(p, p^*) = 1$, then multiplying by $p^*$ in the subscript doesn't change the $p$-valuation, i.e., that $\nu_p(F_{(p^*)^2}) = \nu_p(F_{p^*})$. It is well-known that for all primes $p$ except 5, $p^* \mid p \pm 1$, and hence $\gcd(p, p^*) = 1$.

\begin{lemma}
\label{neclem}
Given an odd prime $p$ not equal to 5, let $\nu_p(F_{p^*}) = k$. Then $\nu_p(F_{(p^*)^2}) = k$.
\end{lemma}

\begin{proof}
$F_{p^*} \mid F_{(p^*)^2}$ since $p^* \mid (p^*)^2$, so $\nu_p(F_{(p^*)^2}) \geq k$. But $\gcd(F_{p^*p}, F_{(p^*)^2}) = F_{\gcd(p^*p, (p^*)^2)}$.

Since $p \not = 5$, we have $p^* \mid p \pm 1$ and hence $\gcd(p^*p, (p^*)^2) = p^*$. So $\gcd(F_{p^*p}, F_{(p^*)^2}) = F_{p^*}$. But then since $p^{k+1} \mid F_{p^*p}$ (by Corollary \ref{neccor}), this implies that $p^{k+1} \nmid F_{(p^*)^2}$, which gives $\nu_p(F_{(p^*)^2}) = k$. $\qed$
\end{proof}

Using Lemma \ref{neclem}, we can show that Conjecture \ref{OpenProb} will not hold if $p^* < p$. The pertinent counterexample, as we saw in the $p = 11$ case with ${100 \choose 10}_F$, is ${(p^*)^2 \choose p^*}_F$. In the base $\mathcal{F}_p = (1, p^*, p^* \cdot p, p^* \cdot p^2, \cdots)$, using $p^* < p$, we have $(p^*)^2 = (0 \ p^*)_{\mathcal{F}_p}$ and $p^* = (0 \ 1)_{\mathcal{F}_p}$. Given this scenario, we have the following theorem. Here we will express ${n \choose k}_F$ as the fraction $\frac{F_n\cdots F_{n-k+1}}{k!_F}$:

\begin{theorem}
\label{nectheo}
Let $p$ be a prime such that $p^* < p$. Then $p \mid {0 \choose 0}_F{p^* \choose 1}_F$ but $p \nmid {(p^*)^2 \choose p^*}_F$.
\end{theorem}

\begin{proof}
Certainly by definition since ${0 \choose 0}_F{p^* \choose 1}_F = 1 \cdot F_{p^*}$, we have $p \mid {0 \choose 0}_F{p^* \choose 1}_F$.

Now $\nu_p({(p^*)^2 \choose p^*}_F) = \nu_p(\frac{F_{(p^*)^2} \cdots F_{(p^*)^2 - p^* + 1}}{F_{p^*} \cdots F_1}) = \nu_p(F_{(p^*)^2} \cdots F_{(p^*)^2 - p^* + 1}) - \nu_p(F_{p^*} \cdots F_1)$.

But for $n = (p^*)^2 - 1, \cdots, (p^*)^2 - p^* + 1$, we have that $p^* \nmid n$, so $p \nmid F_n$, since $\gcd(F_p^*, F_n) = F_{\gcd(p^*, n)} = F_r$ where $r < p^*$. Similarly, by definition of $p^*, p \nmid F_1, F_2, \cdots, F_{p^* - 1}$.

Thus $\nu_p(F_{(p^*)^2} \cdots F_{(p^*)^2 - p^* + 1}) = \nu_p(F_{(p^*)^2})$ and $\nu_p(F_{p^*} \cdots F_1) = \nu_p(F_{p^*})$.

But by Lemma \ref{neclem}, we have $\nu_p(F_{(p^*)^2}) - \nu_p(F_{p^*}) = 0$. So $\nu_p({(p^*)^2 \choose p^*}_F) = 0$, giving $p \nmid {(p^*)^2 \choose p^*}_F$. $\qed$
\end{proof}

Thus for Conjecture \ref{OpenProb} to hold, we must in general have $p^* \geq p$. It is well-known that the largest value $p^*$ can take is $p + 1$, so our conjecture is only interested in 5 and primes $p$ whose $p^*$ attains this maximum. The list of such $p$ is $(2,3,7, 23,43,67,83, 103, \cdots)$. It is not known whether this list is finite or infinite (see \cite{CubRou14}). We can thus assert the necessary conditions placed on $p$ in the conjecture.

\addtocounter{section}{-3}
\addtocounter{theorem}{-2}
\begin{conjecture}
\label{OpenProbRestate}
Let $p^* \geq p$ and let $\mathcal{F}_p$ be as above. Express $(n)_{\mathcal{F}_p} = (n_i)_{i\geq0}$ and $(k)_{\mathcal{F}_p} = (k_i)_{i\geq0}$. Then

\begin{center}
$p \mid {n \choose k}_F \Leftrightarrow p \mid {n_0 \choose k_0}_F {n_1 \choose k_1}_F {n_2 \choose k_2}_F \cdots$.
\end{center}
\end{conjecture}
\addtocounter{section}{3}
\addtocounter{theorem}{2}

Chen and Sagan have shown that Conjecture \ref{OpenProb} holds for $p = 2, 3$, and our Lemma~\ref{WeightLift} handles $p=5$. Preliminary computations show that Conjecture \ref{OpenProb} holds for the first 500 rows of the fibonomial triangle for $p = 7, 23, 43, 67$, and 83. We now discuss further directions of research related to the Fibonomial triangle.

\section{Further Directions}

\subsection{The exact residue class of ${n \choose k}_F$ mod 5}

Extending the divisibility results of Corollary \ref{MainCorollary} to exact residue classes for all ${n \choose k}_F$ will not be possible using only Theorem \ref{KnuWil}, since that theorem relates to divisibility of ${n \choose k}_F$ by a prime $p$, not the exact residue class. However, one can derive simple relations using combinatorial and other number theoretic methods. Take for example the following lemma, which makes use of the fact mentioned in the introduction, that $F_n$ counts the number of tilings of a row of $n-1$ squares by dominos and monominos:

\begin{lemma}
$F_{n+5} \equiv_5 3 \cdot F_n$
\end{lemma}

\begin{proof}
$F_{n + 5}$ counts the number of ways to tile $n + 4$ squares with dominos and monominos. These tilings can be broken up into two cases. Number the lines between squares as $1, 2, ..., n+3$. Mark line $n$. Let the first type of tiling have a domino crossing this line and the second type not have a domino crossing the line. Certainly these cases are disjoint and each tiling falls under one of the cases. The first type is the number of ways to tile the first $n-1$ tiles for each tiling of the last 3 tiles ($F_n \cdot F_4$). The second type is the number of ways to tile the first $n$ tiles for each tiling of the last 4 tiles ($F_{n+1} \cdot F_5$). This gives the relation $F_{n + 5} = F_4F_n + F_5F_{n + 1}$. Replacing the known values, we arrive at the desired result: $F_{n + 5} = 3F_n + 5F_{n + 1} \equiv_5 3F_n$. $\qed$
\end{proof}

This lemma gives immediate rise to the following corollary, which expresses certain fibonomial coefficients in terms of the coefficients five lines higher in the fibonomial triangle:

\begin{corollary}
For $k = 0, 1, 2, 3, 4$, we have ${n + 5 \choose k}_F \equiv_5 3^k{n \choose k}_F$.
\end{corollary}

\begin{proof}
Let $k$ be as above. For $k = 0$ we have $1 \equiv_5 {n + 5 \choose 0} \equiv_5 3^0{n \choose 0} \equiv_5 1$. For the other cases, observe that $k!_F$ has no factors of 5 and thus is invertible (mod 5). Then

\begin{center}
${n + 5 \choose k}_F \equiv_5 \frac{F_{n+5}...F_{n+5-k+1}}{k!_F} \equiv_5 [k!_F]^{-1} \cdot F_{n+5}...F_{n+5-k+1} \equiv_5$
\end{center}

\begin{center}
$[k!_F]^{-1} \cdot 3F_{n}...3F_{n-k+1} \equiv_5 3^k\frac{F_{n}...F_{n-k+1}}{k!_F} \equiv_5 3^k{n \choose k}_F$
\end{center}

and the congruence is shown. $\qed$
\end{proof}

Perhaps similar methods can be used to show a congruence for all ${n \choose k}_F$.

\subsection{When $p^* < p$}

As shown in Section 4, one cannot use the base $\mathcal{F}_p$ to derive a conjecture similar to Conjecture \ref{OpenProb} for primes where $p^* < p$. Another direction to research is to determine whether there is a more appropriate base to choose when this is the case. This seems unlikely, since $\mathcal{F}_p$ is constructed so that for each term in the base, the Fibonacci number with that subscript has a greater $p$-valuation than the Fibonacci number with the previous term as its subscript, and hence seems like the ``correct" base to use when studying the Fibonacci numbers mod $p$. However, perhaps there is a different principle to use in constructing a base when $p^* < p$ which will avoid the problems arising from $\mathcal{F}_p$.

%

\end{document}